\documentclass[12pt,reqno]{amsart}
\usepackage{amscd,amssymb,graphicx,color,a4wide,cite}
\usepackage{listings}
\usepackage{color}
\usepackage{booktabs}
\usepackage{hyperref}

\usepackage{mathrsfs}

\usepackage{epstopdf}

\usepackage[all]{xy}
\let\oldtocsection=\tocsection
 
\let\oldtocsubsection=\tocsubsection
 
\let\oldtocsubsubsection=\tocsubsubsection
 
\renewcommand{\tocsection}[2]{\hspace{0em}\oldtocsection{#1}{#2}}
\renewcommand{\tocsubsection}[2]{\hspace{1em}\oldtocsubsection{#1}{#2}}
\renewcommand{\tocsubsubsection}[2]{\hspace{2em}\oldtocsubsubsection{#1}{#2}}

\newtheorem{theorem}{Theorem}[section]
\newtheorem{lemma}[theorem]{Lemma}
\newtheorem{proposition}[theorem]{Proposition}
\newtheorem{corollary}[theorem]{Corollary}

\theoremstyle{definition}
\newtheorem{definition}[theorem]{Definition}

\newtheorem{example}[theorem]{Example}

\theoremstyle{remark}
\newtheorem{remark}[theorem]{Remark}

\numberwithin{equation}{section}
\numberwithin{figure}{section}

\newcommand{\NN} {\mathbb{N}}
\newcommand{\ZZ} {\mathbb{Z}}

\newcommand{\RR} {\mathbb{R}}
\newcommand{\CC} {\mathbb{C}}

\renewcommand{\AA} {\mathbb{A}}

\newcommand {\shA} {\mathcal{A}}

\newcommand {\shM} {\mathcal{M}}

\newcommand {\shN} {\mathcal{N}}
\newcommand {\shQ} {\mathcal{Q}}

\newcommand {\shP} {\P}

\newcommand {\shX} {\mathcal{X}}

\newcommand {\Aff} {\operatorname{Aff}}

\newcommand {\GL} {\operatorname{GL}}
\newcommand {\gp} {{\operatorname{gp}}}

\newcommand {\Hom} {\operatorname{Hom}}

\newcommand {\id} {\operatorname{id}}

\newcommand {\Int} {\operatorname{int}}

\renewcommand {\ker } {\operatorname{ker}}

\newcommand {\lra} {\longrightarrow}

\newcommand {\M} {\mathcal{M}}

\renewcommand{\O} {\mathcal{O}}

\renewcommand {\P} {\mathscr{P}}

\newcommand {\sing} {\mathrm{sing}}

\newcommand {\Spec} {\operatorname{Spec}}

\newcommand {\ul} {\underline}
\newcommand {\ol} {\overline}

\newcommand {\X} {\mathcal{X}}



\newcommand\restr[2]{{% we make the whole thing an ordinary symbol
  \left.\kern-\nulldelimiterspace % automatically resize the bar with \right
  #1 % the function
  \vphantom{\big|} % pretend it's a little taller at normal size
  \right|_{#2} % this is the delimiter
  }}

\newcommand{\bigslant}[2]{{\raisebox{.2em}{$#1$}\left/\raisebox{-.2em}{$#2$}\right.}}

\makeatletter
\newcommand{\xMapsto}[2][]{\ext@arrow 0599{\Mapstofill@}{#1}{#2}}
\def\Mapstofill@{\arrowfill@{\Mapstochar\Relbar}\Relbar\Rightarrow}
\makeatother

\newtheoremstyle{cited}%
  {3pt}% (space above)
  {3pt}% (space below)
  {\itshape}% (body font)
  {}% (indent amount)
  {\bfseries}% {theorem head font}
  {.}% {punctuation after theorem head}
  {.5em}% {space after theorem head}
  {\thmname{#1} \thmnumber{#2} \thmnote{\normalfont#3}}% {theorem head spec}

\theoremstyle{cited}

\usepackage{tcolorbox}

\usepackage{mathtools}

\usepackage{tikz}
\usetikzlibrary{shapes,arrows}
\usetikzlibrary{positioning}
\usepackage{fixltx2e}

\usepackage{epstopdf}

\begin{document}

%===========================================================
\title[Topological Torus Fibrations on Calabi--Yau Manifolds]{Topological Torus Fibrations on Calabi--Yau Manifolds via Kato--Nakayama Spaces}

% author one information
\author{H\"ulya Arg\"uz} 
%\curraddr{}

\thanks{This project received funding from the Fondation Math\'ematiques Jacques Hadamard.}

\address{H\"ulya Arg\"uz, Labaratoire de Math\'ematiques de Versailles, Universit\'e de Versailles St-Quentin-en-Yvellines, 78000/France}
\email{nuromur-hulya.arguz@uvsq.fr}
%\date{\today}
\maketitle

\begin{abstract}
A large class of complex algebraic varieties admit degenerations into toric log Calabi--Yau spaces, formed by unions of toric varieties glued along toric strata. Such degenerations were introduced by Gross and Siebert as toric degenerations. This paper is an expository article on toric degenerations of Calabi--Yau manifolds and Kato--Nakayama spaces. We first review the combinatorial data used to reconstruct a toric degeneration from a toric log Calabi--Yau space $X_0$, via the Gross--Siebert reconstruction algorithm. We then explain how one can understand the total space on a topological level from the Kato--Nakayama space of $X_0$, which is defined in terms of this combinatorial data. We illustrate through a concrete example, focusing on a degeneration of $K3$-surfaces, that the Kato--Nakayama space of $X_0$ defines a homeomorphism to the total space restricted to the inverse image of a circle on the base. We also investigate torus fibrations on the general fiber by further analysis of the topology of the Kato--Nakayama space. The proofs of the results presented here appear in joint work with Bernd Siebert \cite{AS}. 
\end{abstract}

\tableofcontents
%\bigskip

%===========================================================
%===========================================================

%===========================================================
%\title[Log GW-invariants]{Logarithmic Gromov-Witten Invariants}

\tableofcontents
\bigskip

%===========================================================
%===========================================================
%\section*{Introduction}
%The purpose of this... 

%===========================================================
\section{Introduction}
One of the major conjectures in mirror symmetry is the Strominger--Yau--Zaslow Conjecture \cite{SYZ}, which postulates that there exists dual special Lagrangian torus fibrations on mirror pairs of Calabi--Yau manifolds. After weakening the special Lagrangian assumption, the construction of such torus fibrations has been carried out in some cases in \cite{GrossInv,HZ}. These constructions lead to powerful insights into understanding the topology and geometry of mirror pairs of Calabi--Yau manifolds \cite{GrossTopology, GrossGeometry}. One of the major achievements in this context is an algebro-geometric approach to the SYZ-conjecture developed by Gross and Siebert \cite{Gross-Siebert}. The  essential  motivation  of  this approach  is  to consider a toric degeneration
\begin{equation}
    \label{Eq: pi}
    \pi : \mathcal{X} \to \Spec \CC[t]
\end{equation}
of Calabi--Yau varieties into unions of toric varieties, glued along toric strata. One furthermore imposes these degenerations to be of toroidal nature near the $0$-dimensional toric strata \cite{Gross-Siebert5}. In general, such degenerations are over a discrete valuation ring, which throughout this article we assume is given by \[\AA^1= \Spec \CC [t].\] The idea of the Gross--Siebert program, which we review in \S\ref{Sec: reconstruction}, is then to investigate mirror symmetry around the central fiber of such degenerations. To do this, one uses combinatorial data associated to the central fiber, that is encoded in a \emph{log structure}, given by a sheaf of monoids on it. We discuss log structures along with several examples in \S\ref{Sect: logarithmic geometry}. 

Though one can in principle impose several different log structures on a scheme, the one we focus on to understand the topology of toric degenerations arises as a structure of pairs $(X,D)$, given by a scheme, or a complex analytic space $X$, and a divisor $D\subset X$. The log structure $\M_X$ on $X$ is the sheaf of regular functions on $X$ that are invertible away from $D$. We use the notation
\[  X^{\dagger} = (X,\mathcal{M}_X) \]
for a log space. By a pull-back, the log structure on $X$ induces a log structure on $D$. When studying toric degenerations, we study the pair $(\mathcal{X},\mathcal{X}_0)$, 
formed by the total space and the central fiber $\mathcal{X}_0 \subset \mathcal{X}$. This defines the log structure on the central fiber $\mathcal{X}_0$. If at a general point of the singular locus of $\mathcal{X}_0$ where two irreducible components meet, the local equation inside the total space $\mathcal{X}$ is given as $xy=f\cdot t^e$ for some function $f$, the log structure captures $e\in \mathbb{N}$ and the restriction of $f$ to $x=y=0$.

On $\AA^1$, the base of a toric degeneration, we also consider the  standard log structure on the closed point $\Spec \CC \subset \AA^1$, defined by the pair $(\AA^1,(0))$. The map $\pi$ in \eqref{Eq: pi} lifts to a map between log schemes, and after restriction  to the central fiber, induces a morphism
\begin{equation}
    \label{Eq: pi-0}
    \pi_0^{\dagger}: \mathcal{X}_0^{\dagger} \lra \Spec \CC^{\dagger} 
\end{equation}
A log space formed by a union of toric varieties, along with a morphism $\pi_0^{\dagger}$ to the standard log point, is called a \emph{toric log Calabi--Yau space}, if the log structure is sufficiently nice -- see \cite[Defn 4.3]{Gross-Siebert3}. We can define a \emph{generalised momentum map}
\begin{equation}
\label{Eq: momentum}
    \mu: \mathcal{X}_0 \to B 
\end{equation}
from a toric log Calabi-Yau space, which restricts to the standard toric momentum map on all toric irreducible components of the central fiber. Thus, $B$, referred to as the intersection complex of $\mathcal{X}_0$, is formed by the union of the momentum polytopes corresponding to toric irreducible components of $\shX_0$. We review in \S\ref{Sect: toric varieties} how to describe these momentum polytopes, by defining the affine moment map without a reference to a particular symplectic structure. The attaching maps of these different momentum polytopes are determined by intersection patterns of the different toric irreducible components of $\shX_0$. Under suitable assumptions on these intersection patterns, imposed in the definition of a toric degeneration, $B$ is a topological manifold, which carries the structure of an \emph{integral affine manifold with singularities} as discussed in \S\ref{Sec: reconstruction}, where the singular (discriminant) locus is given by a codim $2$ submanifold $\Delta \subset B$. Moreover, this affine structure is uniquely determined by the log structure on $\mathcal{X}_0$. Conversely, under certain restrictions on the affine monodromy on $B$, a choice of \emph{lifted gluing data}, given by
\[  s \in H^1(\iota_{\star} \breve{\Lambda} \otimes \CC^{\times}) \]
where $\iota:B \setminus \Delta \hookrightarrow B$ is the inclusion map and $\breve{\Lambda}$ is the sheaf of integral cotangent vectors, uniquely determines a toric Calabi--Yau \cite[\S4]{Gross-Siebert3}. The gluing data is the technical heart of the combinatorial data needed to reconstruct toric degenerations from integral affine manifolds with singularities. We review this in \S\ref{Sec: reconstruction}, and for a comprehensive discussion of the gluing data we refer to \cite[\S5.2]{GHS}. 

In general, given a union of toric varieties glued along toric strata, it is a challenging question to determine the obstructions to smoothing it. Moreover, even if the existence is ensured, the uniqueness would not hold in general, as there can be different toric degenerations of different Calabi--Yau manifolds with isomorphic central fibers on a scheme theoretic level. However, one of the major results of the Gross-Siebert program states that once fixing a toric log Calabi--Yau structure, which roughly amounts to fixing an integral affine manifold with singularities along with gluing data, there is a canonical way to smooth it, and to construct the coordinate ring of the total space \cite{Gross-Siebert}. This construction is done by an inductive procedure involving a wall structure on $B$. 

Athough the construction of the coordinate ring of the total space of a toric degeneration using the Gross--Siebert machinery is typically impossible to carry through in practice, many features of the family are already contained in the log structure on $\mathcal{X}_0$. In \cite{AS} we show that indeed essential topological information about the family can be read off from the toric log Calabi--Yau structure, by studying Kato--Nakayama spaces associated to log spaces \cite{KN}. We review the construction of Kato--Nakayama spaces, and the topology on these spaces in \S\ref{par: The Kato Nakayama Space}. We show that for a toric variety $X$, which carries a natural fine log structure defined by the pair $(X,D)$, where $D$ is the toric boundary divisor, the Kato--Nakayama space is rather easy to describe, and is homeomorphic to the product of the associated momentum polytope with a compact torus. For instance, the Kato--Nakayama space of $\AA^1$, which is obtained by an oriented real blow up at the origin, is homeomorphic to $S^1 \times \RR_{\geq 0}$. Since the restriction of it to the closed point $(0) \in \AA^1$ is $S^1$, by functoriality of the construction of Kato--Nakayama spaces there is a natural map 
\begin{equation}
\label{Eq: KN space over X0}    
    \pi_0^{KN}: \mathcal{X}_0^{KN} \lra (0)^{KN} \simeq S^1 
\end{equation}
induced by the morphism \eqref{Eq: pi-0}.

In \cite{NO}, it is shown that the lift of a log morphism between log spaces $X^{\dagger}$ and $Y^{\dagger}$ to the level of Kato--Nakayama spaces, 
\[  X^{KN} \lra Y^{KN}, \]
in good cases, is a topological fiber bundle, that is, a continuous, surjective map, satisfying local triviality. Following \cite{NO}, we prove the following result in \cite[\S 4]{AS}
\begin{theorem}
\label{Thm: AS}
The map $\pi_0^{KN}:
\mathcal{X}_0^{KN}\to  S^1$ in \eqref{Eq: KN space over X0}, is a topological fiber bundle, and $\mathcal{X}_0^{KN}$ is homeomorphic to the restriction of any analytification $\shX^{an}$ of $\shX$ to a sufficiently small circle $S^1 \subset
D$, where $D$ denotes the unit disc in $\CC$. In particular, the restriction of $\pi_0^{KN}$ to a fibre over $\xi\in S^1$, denoted by $\mathcal{X}_0^{KN}(\xi)$, induces a homeomorphism
\begin{equation}
\nonumber
\label{Eq: restriction to a phase}
\shX_t \simeq \mathcal{X}_0^{KN}(\xi)
\end{equation}
between the general fibre of a toric degeneration and $\mathcal{X}_0^{KN}(\xi)$.
\end{theorem}
Theorem \ref{Thm: AS} describes the topology of a toric degeneration, particularly the general fiber up to homeomorphism from the data of a toric log Calabi--Yau space. For similar results in the study of smooth affine hypersurfaces in smooth toric varieties see also \cite{RSTZ}. By studying the topology of the Kato--Nakayama space further in \cite{AS}, we show that that the composition of the natural retraction map $\mathcal{X}_0^{\mathrm{KN}} \to \mathcal{X}_0$ with the generalised momentum map $\mu$ in \eqref{Eq: momentum} is a topological torus fibration on $\mathcal{X}_0^{KN}$, and hence after restricting to $\xi\in S^1$ as in Theorem \ref{Thm: AS}, defines a torus fibration on the general fiber of a toric degeneration. Thus, focusing on Calabi--Yau manifolds which admit a toric degeneration, we obtain a method to construct topological torus fibrations on them. Such fibrations have smooth torus fibers away from the discriminant locus on the base $B$. To analyse the singular fibers we need to investigate closer the construction of the Kato--Nakayama space. We do this in the case of the degeneration of K3 surfaces in \S\ref{sec: example}. 

An advantage of studying torus fibrations on Calabi--Yau manifolds by Kato--Nakayama spaces is that rather than a single topological torus fibration, which is a continuous surjective map from the Calabi--Yau to $B$, whose general fibers are tori \cite{GrossInv}, we obtain a moduli space of Calabi--Yau manifolds and topological torus fibrations on them, parametrised by a choice of gluing data. Moreover, in \cite{AS} we show that this construction can be carried compatibly with real structures, and in this way we obtain topological descriptions of the real loci in Calabi--Yau manifolds, from Kato--Nakayama spaces. This perspective lead to several further conclusions on the topology of the real loci in some three dimensional Calabi--Yau manifolds studied in \cite{AP}, and relations to Hodge-theoretic mirror symmetry \cite{AP2}. We note that though for simplicity in this article we focus on Calabi--Yau manifolds, this approach can also be carried in a more general setup, in the context of varieties with effective anti-canonical class \cite{AS,GHS}. For related work in the context of Fano manifolds see also \cite{Prince}. 

\subsection*{Acknowledgements} I am indebted to Mark Gross and Bernd Siebert for the many useful discussions that provided me invaluable insight, criticism, and encouragement throughout my mathematical journey. I am also grateful to Tom Coates and Dimitri Zvonkine for their support and helpful feedback which improved the exposition of this article. I also thank the organisers of the G\"okova Geometry Topology Conferences for providing a wonderful stimulating atmosphere, and for giving me the opportunity to contribute this paper to the conference proceedings. I am particularly grateful to the referee for the very careful reading and many useful comments.

\section{Monoids and Log Structures}
\label{Sect: logarithmic geometry}
In this section we first review some basics on monoids, and then define log structures. For details we refer to \cite{Ogus,Kato}.

\begin{definition}
A {\it monoid} is a set $\shM$ with an associative binary operation with
a unit. The monoid operation is usually written multiplicatively, in which case we will denote the identity by $1$. A {\it homomorphism of monoids} is a function $\beta:\shP\to\shQ$ between monoids such that
$\beta(1)=1$ and $\beta(p\cdot p')=\beta(p)\cdot\beta(p')$. Throughout this paper we assume all monoids are commutative.
\end{definition}
The basic example of a monoid is the set of natural numbers $\NN$, under the addition operation. 

\begin{definition} 
The {\it Grothendieck group of a commutative monoid $\shM$} is the abelian group generated by $\shM$, denoted by $\shM^{gp}$ defined by
\[\shM^{gp}~:=~\{ \shM \times \shM / \sim \}\]
with the equivalence relation compatible with the monoid operation defined as $(x,y)\sim(x',y')$ if and only if there exists an element $p\in\shM$ such that $pxy' = pyx'$.
\end{definition}
Note that $\shM^{gp}$ is the smallest group containing $\shM$, and there is a natural map from $\shM$ to its associated Grothendieck group $\shM^{gp}$ sending an element $p\in\shP$ to the equivalance class $(p,1)$ denoted by $p/1$. Hence, for any abelian group $G$ and a monoid $\shM$, we have $\mathrm{Hom}_{Mon}(\shM,G) = \mathrm{Hom}_{Ab}(\shM^{gp},G)$. 

\begin{definition}
\label{integral monoid}
A monoid $\shM$ is called {\it integral } if the cancellation law holds (i.e. if $xy =x'y$ then $x = x'$), and it is called {\it fine } if it is finitely generated and integral. We call an integral monoid $\shM$ \emph{saturated}, if whenever $p \in \shM^{\gp}$, and $n$ is a positive integer such that $np \in \shM$ then $p\in \shM$. A monoid $\shM$ is called {\it toric} if it is fine, saturated and $\shM^{gp}$ is torsion free.
\end{definition}

\begin{definition}
Let $X$ be an analytic space with the usual analytic topology, or generally a scheme so that the underlying space is endowed with the \'etale topology. A \emph{pre log structure} on $X$ is a sheaf of monoids $\shM$ on $X$ together with a homomorphism of monoids $\beta :\shM \lra \mathcal{O}_X$ where we consider the structure sheaf $\mathcal{O}_X$ as a monoid with respect to multiplication.
\end{definition}

\begin{definition}
\label{log structure}
A pre log structure on $X$ with the morphism of monoids $\alpha :\shM \lra (\mathcal{O}_X,\cdot)$ is called a \emph{log structure} if $\alpha$ induces an isomorphism\\ $$\restr{\alpha}{\alpha^{-1}(\mathcal{O}_X^{\times})}:{\alpha}^{-1}(\mathcal{O}_X^{\times})\lra \mathcal{O}_X^{\times} $$
We will call a scheme $X$ associated with a log structure $\alpha_X:\shM_x\lra\mathcal{O}_X$ a log scheme and denote it by $(X,\shM_X)$ or by $X^{\dagger}$. We refer to the homomorphism $\alpha_X$ as the structure homomorphism.
\end{definition}

\begin{definition}
Given a log scheme $(X,\shM_X)$ we define $\overline \shM_X:=\bigslant{\shM_X}{\mathcal{O}_X^{\times}}$, and refer to it as the \emph{ghost sheaf} of $(X,\shM_X)$.
\end{definition}

\begin{example}
 Let $X$ be a scheme. Let $\shM_X:= \mathcal{O}_X^{\times}$, and $\alpha_X:\mathcal{O}_X^{\times}\lra \mathcal{O}_X$ 
  be the inclusion. Clearly, this defines a log structure on $X$, called the trivial log structure.
\end{example}

\begin{example}
Let $X:=\Spec \CC$, $\shM_X:= \CC^{\times} \oplus \NN$, and
define $\alpha_X:\M_X \to \CC$ as follows.

\begin{center}

$\alpha_X ( x , n ) := \left\{
	\begin{array}{ll}
		x  & \mbox{if } n = 0 \\
		0 & \mbox{if } n \neq 0
	\end{array}
\right.
$
\end{center}
A point $\Spec \CC$ together with this log structure is called the standard log point.

\end{example}

\begin{example} Let $X:=\Spec \CC$, and 
$\shM_X:=S^1\times \mathbb{R}_{\geq0}$. Define
\begin{eqnarray}
\nonumber
\alpha:S^1\times \mathbb{R}_{\geq0}\rightarrow \mathbb{C}\\
\nonumber
(e^{i\phi},r)\rightarrow re^{i\phi}
\end{eqnarray}
This is by definition a prelog structure on $X$. Observe that  
$$\restr{\alpha}{{\alpha}^{-1}(\mathbb{C}^{\times})}:\mathbb{C}^{\times} \rightarrow \mathbb{R}_{>0}\times S^1 $$
$$z\rightarrow (|z|,arg(z))$$
 is an isomorphism. Indeed we have $(\alpha|_{\alpha^{-1}(\mathbb{C}^{\times})})^{-1}=\mathbb{R}_{>0}\times S^1$. Thus, $\alpha$ gives a log structure on $\Spec \CC$ and together with this log structure $\Spec \CC$ is called the polar log point.
\end{example}

\begin{example} 
\label{The divisorial log structure} 

Let $X$ be a scheme, and $D\subset X$ be a divisor. Let $j:X\setminus D\to X$ be the embedding of the
complement. Define
$$\shM_{(X,D)}:=j_*(\mathcal{O}_{X\setminus D}^\times)\cap \mathcal{O}_X$$ So, $\shM_{(X,D)}$ is the sheaf of regular functions on $X$ with zeroes in $D \subset X$, namely the regular functions on $X$ which are units on $X\setminus D$. We set $\alpha_X:\shM_{(X,D)} \hookrightarrow \mathcal{O}_X$ to be the natural inclusion. Hence we get $\alpha^{-1}(\mathcal{O}_X^{\times})=\mathcal{O}_X^{\times} \subset M_{(X,D)}$, which shows  \[\restr{\alpha}{\mathcal{O}_X^{\times}}:{\alpha}^{-1}(\mathcal{O}_X^{\times})\lra \mathcal{O}_X^{\times}\] is an isomorphism. We refer to  $\shM_{(X,D)}$ as the \emph{divisorial log structure} on $X$.
\end{example}

\begin{remark}
There is a short exact sequence \[0\lra \mathcal{O}_X^{\times}\lra \shM_X^{gp} \lra \bigslant{\shM_X^{gp}}{\mathcal{O}_X^{\times}}\lra 0.\] If we have a log scheme $(X,\shM_{(X,D)})$ equipped with the divisorial log structure then $\bigslant{\shM^{gp}_X}{\mathcal{O}_X^{\times}}$ is the sheaf of Cartier divisors on $X$ with support in $D$.
\end{remark}

\begin{remark}
If $X$ is a scheme with log structure $\shM_X$ then from the definition \ref{log structure} it follows that the only invertible section of $\overline \shM_X:=\bigslant{\shM_X}{\mathcal{O}_X^{\times}}$ is the identity. This means that $\overline \shM_X$ is torsion free.

\end{remark}

\begin{example}
Let $X=\mathbb{A}^2$ be the affine plane with the toric invariant divisor $D= \{xy=0 \}$. The sheaf of monoids $M_{(\mathbb{A}^2,D)}$ on $\mathbb{A}^2$ is the sheaf of regular functions on $\mathbb{A}^2$ which are invertible on $X\setminus D$. So, the stalk of  $\overline \shM_{(\mathbb{A}^2,D)}$ over $0$, $\overline \shM_{(\mathbb{A}^2,D),0}$, is isomorphic to $\mathbb{N}^2$. To generalise, when we have a normal crossing divisor $D=(z_1\cdots z_k=0)\subset X$ then  
\[\overline \shM_{(X,D),0}=\mathbb{N}_{D_1} \oplus \cdots \oplus \mathbb{N}_{D_k},\] 
where $\mathbb{N}_{D_i}$ denotes the constant sheaf $\mathbb{N}$. Furthermore, the stalk over a geometric point $\overline{x}$ is isomorphic to the direct sum of $r$ copies of $\mathbb{N}$ where $r$ denotes the number of components of $D$ containing $\overline{x}$.
\end {example}

\begin{definition}
\label{Log pre log}
Let $\alpha:\shM\lra (\mathcal{O}_X,\cdot)$ be a prelog structure on $X$. One can force a log structure on $X$, by defining 
$$\shM^a:=\bigslant{\shM\oplus \mathcal{O}_X^{\times}}
{\{(p,\alpha(p)^{-1})\,\big|\, p\in \alpha^{-1}(\O^\times_X)\big\}}$$
and setting
$$\alpha^a(p,h)= h\cdot \alpha(p)$$

Let us check that the pair $(\shM^a,\alpha^a)$ is a log structure on $X$. For this, we need to show that the map $\restr{\alpha^a}{(\alpha^a)^{-1}(\mathcal{O}_X^{\times})} $ is an isomorphism. Clearly, $\restr{\alpha^a}{(\alpha^a)^{-1}(\mathcal{O}_X^{\times})} : (\alpha^a)^{-1}(\mathcal{O}_X^{\times}) \to \mathcal{O}_X^{\times}$ is surjective, since for any $a \in \mathcal{O}_X^{\times}$ we have $\alpha^a(1,a)=a$. Let $\overline{(x,a)} \in \ker \restr{\alpha^a}{(\alpha^a)^{-1}(\mathcal{O}_X^{\star})}$ we will show $\overline{(x,a)}=(1,1)$. Here we use multiplicative notation and denote the identity elements of the monoids $\mathcal{O}_X$ and $\shM$ by $1$ and the identity element of $\shM^a$ by $(1,1)$. So, $\alpha^a(\overline{(x,a)})=1 \Rightarrow \alpha(x) \cdot a =1 \Rightarrow (x,a)=(x,\alpha(x)^{-1})$ and $x \in \alpha^{-1}(\O^\times_X)$.
Note that under the equivalence relation on $\shM^a$, two sections $(x,a)$ and $(y,b)$ in $\shM^a$ are equal if there are local sections $\alpha(p)$ and $\alpha(q)$ of $\mathcal{O}_X^{\times}$ such that $(x,a) \cdot (q,\alpha(q)^{-1})=(y,b) \cdot (p,\alpha(p)^{-1})$. Thus, we obtain $(x,a) \sim (1,1)$. Hence, the result follows. 
\end{definition}

\begin{definition}
\label{log morphism}
If $X$ and $Y$ are log schemes with sheafs of monoids $\shM$ and $\shN$, then we define a morphism $(f,h)$ from $(X,\shM)$ to $(Y,\shN)$ so that $f:X\lra Y$ is a morphism of the underlying schemes and $h:f^{-1}\shN\lra \shM$ is a homomorphism of sheafs of monoids where $f^{-1}\shN$ denotes the inverse image of the sheaf $\shN$ so that the following diagram commutes

\begin{displaymath}
    \xymatrix{
        f^{-1}\shN \ar[r]^h \ar[d] & \shM \ar[d] \\
        f^{-1}\mathcal{O}_Y \ar[r]       & \mathcal{O}_X }
\end{displaymath}  
We call $f:=(f,h)$ a {\it morphism of log schemes}.
\end{definition}

\begin{definition}
\label{induced log structure}
Let $Y$ be a log scheme, endowed with a log structure $\alpha_Y:\shM_Y\lra \mathcal{O}_Y$, and let $f:X\to Y$ be a morphism of schemes. Consider the composition 
\[f^{-1}\shM_Y\to f^{-1}\mathcal{O}_Y \to \mathcal{O}_X\] which defines a pre log structure on $X$. The log structure associated to this pre log structure is called the induced log structure or the pull back log structure on $X$ and is denoted by $\shM_X = f^* \shM_Y$.
\end{definition}

Note that if we have a morphism $f:X\to Y$ of schemes where $Y$ is equipped with a log structure $\shM_Y$, for the induced log structure on $X$ we use the notation $f^*\shM_Y$ in Definition \ref{induced log structure}, not to confuse it with $f^{-1}\shM_Y$, as used in Definition \ref{log morphism}.

\begin{definition}
Let $X$ be a scheme, and let $\P$ be a finitely generated monoid. Denote by $\P_X$ the constant sheaf corresponding to $\P$. 
A log structure $\shM$ is called {\it coherent} if \'etale locally on $X$ there exists  and a homomorphism $\P_X\to\mathcal{O}_X$ whose associated log structure is isomorphic to $\shM$. We call $\shM$ {\it integral} if it is a sheaf of integral monoids. If $\shM$ is both coherent and integral then it is called {\it fine }. 
\end{definition}

\begin{definition}
For a scheme $X$ with a fine log structure $\shM$ a {\it chart for $\shM$ } is a homomorphism $\P_X \to \shM$ for a finitely generated integral monoid $\P$ which induces $\P^a\cong \shM$ over an \'etale open subset of $X$. Recall that as $\P_X$ we denote the constant sheaf $P$ on $X$.
For morphism $f:(X,\shM)\to(Y,\shN)$ of schemes with fine log structures a {\it chart for f } is a triple $(\P_X\to\shM, \shQ_Y\to\shN, \shQ\to\P)$ where $\P_X\to\shM$, $\shQ_Y\to\shN$ are charts of $\shM$ and $\shN$ respectively and $\shQ\to\P$ is a homomorphism for which the following diagram commutes.
\begin{displaymath}
    \xymatrix{
        \shQ_X \ar[r]^h \ar[d] & \shM_X \ar[d] \\
         f^{-1}\shN \ar[r]       & \shM }
\end{displaymath}  
\end{definition}
A chart of $f$ also exists \'etale locally. The following proposition is immediate from the definition of a chart.  
\begin{proposition}
\label{toric fine}
Let $X$ be an affine toric variety such that $X= \Spec \CC [\check{\sigma} \cap N]$. Let D be the canonical divisor on $X$ and let $\shM_{X,D}$ be the corresponding divisorial log structure. Then $\shM_{X,D}$ is fine, and a chart for $\shM_{X,D}$ is given by $\check{\sigma} \cap N \to \CC [\check{\sigma} \cap N]$.
\end{proposition}
\begin{example}
\label{toric}
Let $X=\bigslant{\mathrm{Spec}~~\mathbb{C}[x,y,w,t]}{(xy-wt)}$. The associated toric  fan for $X$ has ray generators \[\RR_{\geq0} (0,1,0)+\RR_{\geq0} (-1,0,1)+\RR_{\geq0} (0,-1,1)+\RR_{\geq0} (1,0,0).\] Hence, we have the following toric invariant divisors
$$ D_1=(y=w=0), ~~~~  D_2=(x=w=0), ~~~~ D_3=(x=t=0), ~~~~ D_4=(y=t=0)  $$
Let $D$ be the canonical divisor, that is, $D = -D_1 - D_2 - D_3 - D_4 $. Denote by $s_\star$ the sections of
a log structure $\mathcal{M}_{X,D}$, defined by the monomial functions indicated in the subscript. Note that sections $\Gamma(X,\shM_{X,D})$ of the log structure $\shM_{X,D}$ are $s_x,~s_y,~,s_w,~s_t$. Clearly these are regular functions having zeroes in $D$. A chart for $\shM_{X,D}$ is given by
\[ \shM= \langle e_1,\cdots e_4~|~e_1+e_2=e_3+e_4 \rangle \stackrel{\phi}{\longrightarrow}\Gamma(X,\shM_{X,D}) \]
where
\[ \phi(e_1)= x, \,\ \phi(e_2)= y, \,\ \phi(e_3)= w,\,\phi(e_4)= t \]

\end{example}

\begin{definition}
\label{strict}
A morphism $f:X\to Y$ between log schemes $(X,\shM_X)$ and $(Y,\shM_Y)$ is called {\it strict} if $f$ induces an isomorphism between $\shM_X$ and $f^*\shM_Y$.
\end{definition}

\begin{example}
Let $(X,\shM_X)$ be a fine log scheme with a chart $\shM\to \Gamma(X,\shM_X)$. Then $X\to \Spec (\CC [\shM])$ is a strict morphism of analytic spaces. Indeed, a chart $\shM\to \Gamma(X,\shM_X)$ is equivalent to a morphism $X\to \Spec (\CC [\shM])$.
\end{example}

\section{Toric Varieties and Affine Moment Maps}
\label{Sect: toric varieties}
Throughout this section, we assume basic familiarity with toric geometry \cite{Fulton}. We fix the lattices $M=\ZZ^n$, $N=\Hom_{\ZZ}(M,\ZZ)$ and we will denote by $M_{\RR}=M\otimes_{\ZZ}
\RR$, $N_{\RR}=N\otimes_{\ZZ} {\RR}$ the associated real vector spaces. If
$\sigma\subset N_\RR$ is a cone then the set of monoid homomorphisms
$\sigma^\vee= \Hom(\sigma,\RR_{\ge0}) \subset M_\RR$ denotes its
dual cone. A \emph{lattice polyhedron} is the intersection of
rational half-spaces in $M_\RR$ with an integral point on each
minimal face. Let $\Xi\subset M_\RR$ be a
full-dimensional, convex lattice polyhedron. Let $X$ be the
associated complex toric variety. A basic fact of toric geometry
states that the fan of $X$ agrees with the normal fan $\Sigma_\Xi$
of $\Xi$. From this description, $X$ is covered by affine toric
varieties $\mathrm{Specan}\CC[\sigma^\vee\cap M]$, for $\sigma\in
\Sigma_\Xi$. Since the patching is monomial, it preserves the real
structure of each affine patch. Hence the real locus
$\Hom(\sigma^\vee,\RR)\subset \Hom(\sigma^\vee,\CC)$ of each affine
patch glues to the real locus $X_\RR\subset X$. Unlike in the
definition of $\sigma^\vee$, here $\RR$ and $\CC$ are
multiplicative monoids. Moreover, inside the
real locus of each affine patch there is the distinguished subset
\begin{equation}
\label{eq: positive real locus}    
    \sigma=\Hom(\sigma^\vee,\RR_{\ge 0})\subset\Hom(\sigma^\vee,\RR),
\end{equation}
with ``$\Hom$'' referring to homomorphisms of monoid. These also
patch via monomial maps to give the \emph{positive real locus}
$X_{\ge0}\subset X_\RR$. Having introduced the positive real locus $X_{\ge0}\subset X_\RR$ we
are in position to define abstract momentum maps.

\begin{definition}
\label{Def: Momentum map}
Let $X$ be the complex toric variety defined by a full-dimensional lattice
polyhedron $\Xi\subset M_\RR$. Then a continuous map
\[
\mu: X\lra \Xi
\]
is called an \emph{(abstract) momentum map} if the following holds.
\begin{enumerate}
\item
$\mu$ is invariant under the action of $\Hom(M,U(1))$ on $X$.
\item
The restriction
of $\mu$ maps $X_{\ge0}$ homeomorphically to $\Xi$, thus defining a
section $s_0: \Xi\to X$ of $\mu$ with image $X_{\ge0}$.
\item
The map
\begin{equation}
\label{Eqn: torus times section}
\Hom(M,U(1))\times \Xi\lra X,\quad
(\lambda,x)\longmapsto \lambda\cdot s_0(x)
\end{equation}
induces a homeomorphism $\Hom(M,U(1))\times \Int(\Xi) \simeq
X\setminus D$, where $D\subset X$ is the toric boundary
divisor.
\end{enumerate}
\end{definition}

Projective toric varieties have a momentum map, see e.g.\ \cite{Fulton}, \S4.2.
For an affine toric variety $\mathrm{Specan}\CC[P]$, momentum maps also exist. One
natural construction discussed in detail in \cite{NO}, \S1, is a simple formula
in terms of generators of the toric monoid $P$ (\cite{NO}, Definition~1.2
and Theorem~1.4). Some work is however needed to show that if
$P=\sigma^\vee\cap M$, then the image of this momentum map is the cone
$\sigma^\vee$ spanned by $P$. We give here another, easier but somewhat ad hoc
construction of a momentum map in the affine case.

\begin{proposition}
\label{Prop: Momentum map for affine toric}
An affine toric variety $X=\mathrm{Specan} \CC[\sigma^\vee\cap M]$ has a 
momentum map with image the defining rational polyhedral cone
$\sigma^\vee\subset M_\RR$. 
\end{proposition}

\begin{proof}
If the minimal toric stratum $Z\subset X$ is of dimension $r>0$, we
can decompose $\sigma^\vee\simeq C+ \RR^r$ and acordingly $X\simeq \ol
X\times(\CC^*)^r$ with  $\ol X$ a toric variety with a
zero-dimensional toric stratum. The product of a momentum map $\ol
X\to C$ with the momentum map for $(\CC^*)^r$,
\[
(\CC^*)^r\lra \RR^r,\quad
(z_1,\ldots,z_r)\longmapsto \big(\log|z_1|,\ldots,\log|z_r|\big)
\]
is then a momentum map for $X$. We may therefore assume that $X$ has
a zero-dimensional toric stratum, or equivalently that $\sigma^\vee$
is strictly convex.

Now embed $X$ into a projective toric variety $\tilde X$ and let $\mu: \tilde
X\to \Xi$ be a momentum map mapping the zero-dimensional toric stratum of $X$ to
the origin, which is hence necessarily a vertex of $\Xi$. Then the
cone in $M_\RR$ spanned by $\Xi$ equals $\sigma^\vee$. By replacing $\Xi$ with
its intersection with an appropriate affine hyperplane and $\tilde X$ by the
corresponding projective toric variety, we may assume that $\Xi$ is the convex
hull of $0$ and a disjoint facet $\omega\subset\Xi$. Then
$X=\mu^{-1}(\Xi\setminus \omega)$. To construct a momentum map for $X$ with
image $\sigma^\vee$ let $q: M_\RR\to\RR$ be the quotient by the tangent space
$T_\omega\subset M_\RR$. Then $q(\Xi)$ is an interval $[0,a]$ with $a>0$. Now
$f(x)=x/(a-x)$ maps the half-open interval $[0,a)$ to $\RR_{\ge0}$. A momentum
map for $X$ with image $\sigma^\vee$ is then defined by
\[
z\longmapsto (f\circ q)\big(\mu(z)\big) \cdot \mu(z).
\]
\end{proof}

\section{Toric Degenerations from Combinatorial Data}
\label{Sec: reconstruction}
We will first recall the definition of a toric degeneration, as introduced in \cite{Gross-Siebert5}. We then describe the combinatorial data used to reconstruct such degenerations of Calabi--Yau varieties using the Gross--Siebert program \cite{Gross-Siebert}. 
\begin{definition}
\label{Def: toric degeneration}
Let $R$ be a discrete valuation ring with closed point $0$. A \emph{toric degeneration}  
\[
\pi\colon \shX\to \Spec R
\]
is a flat family whose generic fibre is a normal algebraic space and central fibre $\mathcal{X}_0$ is a union of toric varieties glued along toric boundary strata, so that around the zero dimensional strata the morphism $\pi$ is toroidal. Moreover, we assume $\shX$ polarized, and furthermore require that there exists a closed subset $Z \subset \shX$ of relative codimension at least two, not containing any toric strata of $\mathcal{X}_0$, such that every point in $\shX \setminus Z$ has a neighbourhood which is  \'etale locally isomorphic to an affine toric variety. We refer to $Z$ as the \emph{log-singular locus} on $\shX$.
\end{definition}
We will assume the central fibre of a toric degeneration is projective, which can be achieved under mild assumptions -- see \cite[Thm 2.34]{Gross-Siebert3}. We furthermore assume that it has no self-intersections, unlike in \cite{Gross-Siebert3} for simplicity. 

To a toric degeneration we can associate the data given by a tuple
\begin{equation}
\label{Eq: combinatorial data}    
    (B,\P,\varphi,s)
\end{equation}
where $B$ is an \emph{integral affine manifold with singularities}, $\P$ is a \emph{polyhedral subdivision}, $\varphi$ is a \emph{multi-valued piecewise-linear (MPL) function} and $s$ is \emph{gluing data}. We will review the definitions and discuss this data in more detail in the next subsections. Before doing this, we note that this combinatorial data is the initial data one uses to construct a toric degeneration of a Calabi--Yau using the Gross--Siebert program \cite{Gross-Siebert}. For a significant generalisation of this construction to the case of varieties with effective anti-canonical class see \cite{GHS}. 
\subsection{The Intersection Complex $B$ with a polyhedral decomposition $\P$}
\label{subsect: The intersection complex}
Let 
\[\pi \colon \shX\to \Spec R\] be a toric degeneration as in Definition \ref{Def: toric degeneration}. Since $R$ is a discrete valuation ring, the requirement that near each zero-dimensional toric stratum of $\mathcal{X}_0$, \'etale locally $\pi$ is isomorphic
to a monomial map of toric varieties amounts to describing $\X$ \'etale locally as $\Spec\CC[P]$ with $P$ a
toric monoid and $f$ by one monomial $t=z^{\rho_P}$, $\rho_P\in P$. Under
these conditions it turns out that the generic fibre $\X_\eta$ is a Calabi-Yau
variety \cite{Gross-Siebert3}.

The toric
irreducible components of $\mathcal{X}_0$ are glued pairwise along toric strata in such a
way that the \emph{dual intersection complex} is a closed topological manifold, of the
same dimension $n$ as the fibres of $\pi$. In particular, the notion of toric
strata of $\mathcal{X}_0$ makes sense. The \emph{intersection complex} is defined as follows.
\begin{definition}
\label{def: polyhedral dec on B}
Asuming that $\mathcal{X}_0$ is projective, let $\P$ be the set of the images if momentum
maps, defined as in Definition \ref{Def: Momentum map}, of the toric strata, and let $\P_{\mathrm{max}}\subset\P$ be the maximal elements under
inclusion. For a cell $\tau\in\P$ we denote by $X_\tau\subset \mathcal{X}_0$ the corresponding
toric stratum. The cell complex
\[B=\bigcup_{\sigma\in\P_{\mathrm{max}}} \sigma\] with attaching maps defined by the intersection patterns of the toric strata is called the \emph{intersection complex} 
$(B,\P)$ or \emph{cone picture} of the polarized
central fibre $\mathcal{X}_0$. We refer to the collection $\P$ of closed subsets of $B$ as a \emph{polyhedral decomposition} of $B$.
\end{definition}

We require a polyhedral decomposition $\P$ to satisfy some compatibility conditions, to be able to control how cells of $\P$ interact with the singular set of the affine
manifold $B$. These conditions are explained in detail in Definition $1.22$ in \cite{Gross-Siebert3}.

\begin{remark}
Unlike in \cite{Gross-Siebert3}, for simplicity of notation we assume that no irreducible
component of $\mathcal{X}_0$ self-intersects. On the level of the cell complex $(B,\P)$,
this means that for any $\tau\in\P$ the map $\tau\to B$ is injective. 
\end{remark}  
The barycentric subdivision of $(B,\P)$ is canonically isomorphic to the
barycentric subdivision of the dual intersection complex of $\mathcal{X}_0$, as simplicial
complexes. Thus $B$ is a topological manifold. The following is Proposition~3.1 in \cite{RS}.
\begin{proposition}
\label{prop: RS}
There is a generalized momentum
map $\mu: \mathcal{X}_0\to B$ that restricts to the momentum maps $X_\tau\to \tau$ on
each toric stratum of $\mathcal{X}_0$.
\end{proposition}

The interpretation of the cells of $\P$ as momentum polyhedra endows $B$ with
the structure of an \emph{integral affine manifold} on the interiors of the
maximal cells, that is, a manifold with a coordinate atlas, with transition functions in 
\[\Aff(\ZZ^n)=
\ZZ^n\rtimes \GL(n,\ZZ).\] 
On such manifolds it makes sense to talk about
integral points as the preimage of $\ZZ^n$ under any chart, and they come with a
local system $\Lambda$ of integral tangent vectors. An important insight is that
the log structure on $\mathcal{X}_0$ provides a canonical extension of this affine
structure over the complement in $B$ of the amoeba image $\shA:=\mu(Z)$
of the log singular locus $Z\subset (\mathcal{X}_0)_\sing$ under the generalized
momentum map $\mu:\mathcal{X}_0\to B$. We refer to 
\[ \mu(Z) = \shA \]
as the (amoeba image, or thickening of the) discriminant locus, which we denote by $\Delta$ in $B$. We will see in a moment that away from this locus, $B$ carries the structure of an integral affine manifold. Also to ensure that a toric log Calabu--Yau space $\mathcal{X}_0$ whose intersection complex is $B$, is smoothable, in \cite{Gross-Siebert3} there are conditions imposed on the monodromy around the discriminant locus. Namely, we require $B$ to have \emph{simple singularities}, which we review shortly -- for details see \cite[Defn 1.60]{Gross-Siebert3}. 

We locally describe the affine structure at a codimension one cell
$\rho=\sigma\cap\sigma'$ as follows. The affine structure of the adjacent maximal cells
$\sigma,\sigma'$ already agree on their common face $\rho$. So the extension at
$x\in\Int \rho\setminus\shA$ only requires the identification of
$\xi\in\Lambda_{\sigma,x}$ with $\xi'\in\Lambda_{\sigma',x}$, each complementary
to $\Lambda_{\rho,x}$. Let $X_\sigma$ and $X_{\sigma'}$ be the toric irreducible components of $\mathcal{X}_0$ corresponding to $\sigma$ and $\sigma'$ respectively. A local description around the point $x$ on the total space, is given by the equation $uv=f\cdot t^{\kappa_\rho}$. We
have $u|_{X_\sigma}=z^m$, $v|_{X_{\sigma'}}=z^{m'}$ by the assumption on $u,v$
to be monomial on one of the adjacent components $X_\sigma$, $X_{\sigma'}$.
Now since $\mu^{-1}(x)\cap Z=\emptyset$, the restriction
$f|_{\mu^{-1}(x)}$ yields a map $\mu^{-1}(x)\to \CC^*$. The homotopy class of
this map defines an integral tangent vector $m_x\in \Lambda_\rho$. One then
takes $\xi=m$, $\xi'=-m'+m_x$. See \cite{RS}, \S~2.2 for details. This defines the integral affine structure on $B\setminus\shA$, away from
codimension two cells by the following lemma.

\begin{lemma}
The integral affine structure on the interiors of the maximal cells
$\sigma\in\P$ and at points of $\Int\rho\setminus\shA$ for all
codimension one cells $\rho$ extends uniquely to $B\setminus\shA$.
\end{lemma}

\begin{proof}
Uniqueness is clear because the extension is already given on an
open and dense subset.

At a vertex $v\in B$ we have $\mu^{-1}(v) = X_v$, a zero-dimensional toric
stratum. Let $U\to \mathrm{Specan}\CC[P]$ with $P=K\cap \ZZ^{n+1}$ and $t=z^{\rho_P}$,
$\rho_P\in P$, be a toric chart for $\pi:\X\to \mathrm{Specan} R$ at $X_v$. Here $K$
is an $(n+1)$-dimensional rational polyhedral cone, not denoted $\sigma^\vee$ to
avoid confusion with the cells of $B$. There is then a local identification of
$\mu$ with the composition
\[
\mu_v: \mathrm{Specan}\CC[P]\stackrel{\mu_P}{\lra} K \lra
\RR^{n+1}/\RR\cdot \rho_P
\]
of the momentum map for $\mathrm{Specan}\CC[P]$ with the projection from the cone $K$
along the line through $\rho_P$. Since $\rho_P\in \Int K$, this map projects
$\partial K$ to a complete fan $\Sigma_v$ in $\RR^{n+1}/\RR\cdot \rho_P$. The
irreducible components of $\mathcal{X}_0$ containing $X_v$ have affine toric charts given
by the facets of $K$. Thus this fan describes $\mathcal{X}_0$ at $X_v$ as a gluing of
affine toric varieties. Now any momentum map $\mu$ of a toric variety provides
an integral affine structure on the image with $R^1\mu_*\ul\ZZ$ the sheaf of
integral tangent vectors on the interior. In the present case, this argument
shows first that the restriction of $R^1{\mu_v}_*\ul\ZZ$ to the interior of each
maximal cone $K'\in \Sigma_v$ can be canonically identified with the sheaf of
integral tangent vectors $\Lambda$ on the interiors of maximal cells of $B$.
Second, the argument shows that $R^1{\mu_v}_*\ul\ZZ$ restricted to $\Int K'$ can
be identified with the (trivial) local system coming from the integral affine
structure provided by $\ZZ^{n+1}/\ZZ\cdot\rho_P$. The fan thus provides an
extension of the sheaf $\Lambda$ over a neighbourhood of $v$ and hence
also of the integral affine structure. A possible translational part in the
local monodromy does not arise by the given gluing along lower dimensional
cells.

For any $\tau\in\P$, the extension at the vertices of $\P$ provides also the
extension on any connected component of $\tau\setminus\shA$ containing a vertex.
If $\shA\cap\tau$ has connected components not containing a vertex, one can in
any case show the existence of a toric model with fan $\partial K/\RR\cdot
\rho_P$ of not necessarily strictly convex rational polyhedral cones. The
argument given at a vertex then works analogously.
\end{proof}

We require the affine structure intersection complex $B$ to have  \emph{simple singularities}. This notion of simplicity has been introduced in \cite{Gross-Siebert3} as an indecomposability condition on the local affine monodromy around the singular
locus $\Delta\subset B$ of the affine structure on the dual intersection complex
of $(\mathcal{X}_0,\M_{\mathcal{X}_0})$. We discuss this in more detail in the remaining part of this section.

We first review how to encode the monodromy around the discriminant locus in $B$, before describing simplicity. Let $\omega \in \P^{}[1]$, be the set of one-dimensional
cells of the polyhedral decomposition $P$, as defined in Definition \ref{def: polyhedral dec on B}, and let $\rho \in \P^{[n-1]}$ denote the set of codimension one cells of $\P$. Note that there are two maximal cells $\sigma_\pm$ containing $\rho$, and that there are two vertices $v_\pm$ adjacent to $\omega$. Consider a loop in $B$ with base point $v_-$ and tracing around $\Delta$, by going through $\sigma_-$ followed by $v_+$ and $\sigma_+$, and coming back to $v_-$. We obtain an associated monodromy transformation \begin{equation}
\label{eq: monodromy}
T_{\omega_\rho}(m)=m+ \kappa_{\omega_\rho} \langle \check{d}_\rho, m \rangle d_\omega
\end{equation}
as shown in\cite[\S $1.5$]{Gross-Siebert3}, where $d_\omega \in \Lambda_v$ and $\check{d}_\omega \in \Lambda^*_v$ are the primitive integral vectors pointing from $v_-$ to $v_+$
%and larger on σ+ than σ− 
respectively, and $\kappa_{\omega_\rho} \in \ZZ$ is a constant independent of the choices of $v_\pm$ and $\sigma_\pm$. 
\begin{definition}
\label{def: positive}
We say $(B,\P)$ is \emph{positive} if $\kappa_{\omega_\rho} \geq 0$ for all pairs $\omega \subset \rho$ in \eqref{eq: monodromy}.
\end{definition}

More generally, consider two arbitrary vertices $v, v'$ contained in an $(n-1)$-cell $\rho$. Then, one finds the monodromy transformation takes the form
\[ m \mapsto m + \langle m, \check{d}_\rho \rangle m_{vv'}^\rho  \]
for a well-defined $m_{vv'}^\rho \in \Lambda_\rho$. In the positive case, one can assemble this data as the
monodromy polytope for $\rho$,
\[ \Delta(\rho) = \mathrm{conv}\{m_{vv'}^\rho ~ | ~ v' \in \rho\}.\]
where $v \in \rho$ is a fixed vertex, and a different choice of $v$ leads to a translation of $\Delta(\rho)$. Analogously, for a $1$-cell $\omega$ contained in two arbitrary maximal cells $\sigma,\sigma'$, one obtains
monodromy of the form
\[ m \mapsto m + \langle m, n_\omega^{\sigma\sigma'}  \rangle d_\omega  \]
for some $n_\omega^{\sigma\sigma'} \in \Lambda_\omega^\perp \subset
 \breve{\Lambda}_x$ for any $x \in \Int(\omega) \setminus \Delta$. We can then define
\[ \check{\Delta}(\omega) = \mathrm{conv}\{n_\omega^{\sigma\sigma'} ~ | ~ \omega \in \sigma'\}.\]
This is again well-defined up to translation. For an arbitrary $\tau \in \P$ with $1 \leq \dim \tau \leq \dim B - 1,~ \mathrm{for} ~ \omega \subset \tau$ with $\dim \omega = 1$ and $\tau \subset \rho$ with $\dim \rho = \dim B - 1$, define
\begin{eqnarray}
\nonumber
\Delta_\rho(\tau) & = & \mathrm{conv} \{  m_{vv'}^\rho~|~v' \in \tau \} \\
\nonumber
\breve{\Delta}_\rho(\tau) & = & \mathrm{conv} \{  n_\omega^{\sigma\sigma'}~|~ \tau \subset \sigma' \in \P^{max} \}
\end{eqnarray}
where $v, \sigma$ are fixed with $v \in \tau$ , $\tau \subset \sigma$. Now, we are ready to define the notion of simplicity of $(B,\P)$ \cite[Defn $1.60$]{Gross-Siebert3}.

\begin{definition}
\label{def: simple}
Suppose $(B,\P)$ is positive, as defined in Definition \ref{def: positive}. Then we say $(B,\P)$ is \emph{simple} if for every $\tau \in \P$ with $1 \leq \dim \tau \leq n - 1$, where $n = \dim B$, the following condition holds: Set
\begin{eqnarray}
\nonumber
P_1(\tau) & := & \{\omega \subset \tau ~ | ~ \dim \omega = 1\}, \\
\nonumber
P_{n-1}(\tau) & := & \{\tau \subset \rho | \dim \rho = n - 1\}.
\nonumber
\end{eqnarray}
Then there exists disjoint subsets
\begin{eqnarray}
\nonumber
\Omega_1,\ldots, \Omega_p & \subseteq & P_1(\tau), \\
\nonumber
R_1, \ldots, R_p & \subseteq & P_{n-1}(\tau )
\end{eqnarray}
for some $p \geq 0$ such that
\begin{itemize}
\item[(1)] For $\omega \in P_1(\tau)$ and $\rho \in P_{n-1}(\tau ),~ \kappa_{\omega_\rho} = 0$, where $\kappa_{\omega_\rho}$ is defined as in \eqref{eq: monodromy}, unless $\omega \in \Omega_i,~ \rho \in R_i$ for some $i$.
\item[(2)] For each $1 \leq i \leq p$, $\check{\Delta}_\omega(\tau)$ is independent (up to translation) of $\omega \in \Omega_i$, and we denote $\check{\Delta}_\omega(\tau)$ by $\check{\Delta}_i$ for $ \omega \in \Omega_i$; similarly, $\Delta_\rho (\tau)$ is indepedent (up to
translation) of $\rho \in R_i$, and we denote $\Delta_\rho (\tau)$ by $\Delta_i$.
\item[(3)] Let $e_1, \ldots, e_p$ be the standard basis of $\ZZ^p$, and set local monodromy polytopes
\begin{eqnarray}
\nonumber
\check{\Delta}(\tau) & := & \mathrm{conv} \bigg( \bigcup_{i=1}^p \check{\Delta}_i \times \{e_i \} \bigg) \subseteq (\Delta_\tau^\perp \oplus \ZZ^p) \otimes \RR  \\
\nonumber
\Delta(\tau) & := & \mathrm{conv} \bigg( \bigcup_{i=1}^p \Delta_i \times \{e_i \} \bigg) \subseteq (\Delta_\tau^\perp \oplus \ZZ^p) \otimes \RR. 
\end{eqnarray}
Then we further require $\check{\Delta}(\tau)$ and $\Delta(\tau)$ to be elementary simplices, that is,
simplices whose only integral points are their vertices.
\end{itemize}
\end{definition}
Note that simplicity implies local rigidity of the singular locus of the log
structure as needed in the smoothing algorithm \cite[Defn 1.26]{Gross-Siebert}, but unlike local rigidity, being simple imposes conditions in
all codimensions. In the case of simple singularities, we can determine the log singular locus in the central fibre uniquely, by the following Corollary~5.8 of \cite{Gross-Siebert3}.
\begin{proposition}
If $B$ is an integral affine manifold with simple singularities, then the log singular locus $Z\subset \mathcal{X}_0$ on $\mathcal{X}_0$ with associated intersection complex $B$, is determined uniquely.
\end{proposition}

\subsection{The MPL-function $\varphi$}
\label{Subsect: The MPL function}
We require an additional combinatorial piece of data, given by a multi-valued piecewise linear (MPL) function $\varphi$ defined on $B_0 := B \setminus \Delta$. This data is viewed as specified data, in relation with a choice of polarization in \cite{Gross-Siebert}, whereas in \cite{GHS} it canonically occurs in the presented mirror construction. For a more comprehensive study of the appearence of the (MPL) function see \cite[\S 1]{GHS} (for our purposes we assume the MPL functions in this paper take values in $Q=\NN$). Note that the maximum domains of linearity of $\varphi$ correspond to maximal cells of the polyhedral decomposition $\P$ on $B$. Moreover, this function is uniquely described by one
integer $\kappa_\rho$ on each codimension one cell $\rho\in\P$. If $uv=f\cdot
t^\kappa$ is the local description of $\X$ at a general point of $X_\rho$,
then $\kappa_\rho=\kappa$. Each codimension two cell $\tau$ imposes a linear
condition on the $\kappa_\rho$ for all codimension one cells
$\rho\supset\tau$ assuring the existence of a local single-valued representative
of $\varphi$ in a neighbourhood of $\tau$ (see \cite{GHS}, Example~1.11). Note
that the local representative $\varphi$ is only defined up to a linear function.
Therefore, globally $\varphi$ can be viewed as a \emph{multi-valued piecewise linear
function}, a section of the sheaf of pieceweise linear functions modulo linear
functions. 

\subsection{Lifted Open Gluing Data $s$}
\label{subsec: Open Gluing Data}
To reconstruct a toric degeneration out of the combinatorial data in \cite{Gross-Siebert}, naively one defines a monoid ring associated to each vertex $v \in \P$, by consideration of integral points on the upper convex hull of $\varphi$. An affine cover for the total space of the toric degeneration is then obtained by the spectra of these monoid rings. The technical difficulty arises while determining how to glue the affine open sets forming this cover. However, this gluing is not entirely canonical, and depends on the a choice of \emph{gluing data}. Roughly speaking, gluing data is the data determining how the big torus orbits on the toric irreducible components of the central fibre are assembled together. This assembling is determined by \emph{closed gluing data} -- see \cite{Gross-Siebert3}, Definition~2.3
and Definition~2.10. While closed gluing
data can be interpreted as changing the closed embeddings defined by the inclusion of toric strata on the central fibre of a toric degeneration, we also have a modification of it given by \emph{open gluing
data} \cite{Gross-Siebert,Gross-Siebert3}, where ``Open'' refers
to the fact that these gluing data modify open embeddings. 
We note that the reconstruction of a toric degeneration can be carried either on the \emph{cone side} (the intersection complex) or on the \emph{fan side} (the dual intersection complex) \cite[\S2]{Gross-Siebert3}. For the following discussion in the remaining part of this section we assume we are in the latter case, and define the open gluing data for the fan picture as in \cite[\S2]{Gross-Siebert3}. For the analogous construction of open gluing data on the fan side we refer to \cite[Defn 1.18]{Gross-Siebert}. 

\begin{definition}
\label{Def: PM}
Let $\sigma \in \mathcal{P}_{\mathrm{max}}$ be a maximal cell, and let $\tau \subseteq \sigma$. Define $\breve{\mathrm{PM}}(\tau)$ to be the set of tuples $(s_v)_{v\in \tau}$ where $s_v \in \Lambda_{\sigma} \otimes_{\ZZ} \CC^{\star}$, and the following condition is satisfied: For any face $\omega \subseteq \tau$ and vertices $v, w \in \omega, s_v \cong s_w ~\mathrm{mod}~ \Lambda_{\omega} \otimes_{\ZZ} \CC^{\star}$ , where $\Lambda_{\omega}$ is
viewed as a subspace of $\Lambda_{\sigma}$ via parallel transport along a path from $\omega$ into $\sigma$.
\end{definition}
The following is Definition $2.25$ in \cite{Gross-Siebert3}, where the notation $e\in\Hom(\tau,\sigma)$ denotes the inclusion $e\colon \tau \to \sigma$ of the faces $\tau,\sigma \in \P$. 
\begin{definition}
\label{open_gluing_data}
\emph{Open gluing data} for $\P$ are data
$s=(s_e)_{e\in\coprod\Hom(\tau,\sigma)}$, with
$s_e\in \breve{\mathrm{PM}}(\tau)$ for $e\colon \tau \to \sigma$. They must satisfy
\begin{enumerate}
\item $s_{\id_{\tau}} = 1$ for every $\tau \in \P$.
\item If $e \in \Hom(\tau, \tau')$, $f \in \Hom(\tau', \tau'')$ then $s_{f\circ e}  = s_f \cdot s_e$ wherever defined:
\[s_{f\circ e,\sigma} = s_{f,\sigma} \cdot s_{e,\sigma} ~ \mathrm{for ~ all}~ \sigma \in \P_{\mathrm{max}}  ~\mathrm{with} ~ \sigma \supseteq \tau''.\]
\end{enumerate}
\end{definition}
We note that Definition \ref{open_gluing_data} does not depend on the choice of the maximal $\sigma$, since for any $\sigma, \sigma'\supseteq \tau$  , we identify $(s_v)_{v\in \tau}$ in the description of $\breve{\mathrm{PM}}(\tau)$ in Definition \ref{Def: PM} with the choice $\sigma$ with $(s'_v)_{v\in \tau}$ in the description of $\breve{\mathrm{PM}}(\tau)$ with the choice $\sigma'$, by taking $s'_v$
to be the parallel transport of $s_v$ via a path from the
interior of $\sigma$ through $v$ to the interior of $\sigma'$. In the case when $B$ is positive and simple, we can recover open gluing data from lifted gluing data, defined as follows.
%Rephrasing results in \cite[\S $4$]{Gross-Siebert3}, we conclude that open gluing data parametrizes log structures on the central fibre of a toric degeneration. Moreover, open gluing data can be recovered from \emph{lifted gluing data} \cite[Def 5.1]{Gross-Siebert3}.
\begin{definition}
\label{def: lifted gluing}
Define for $\tau \in \P$ the open subset $W_\tau \subseteq B$ to be the union of interiors of all
cells of the barycentric subdivision $\mathrm{Bar}(\P)$ intersecting $\Int (\tau )$, i.e., having the barycenter of $\tau$ as a vertex. For $e : \omega \to \tau$ write $W_e := W_\omega \cap W_\tau$. Letting $\iota : B_0=B \setminus \Delta \hookrightarrow B$ be the inclusion, we define \emph{lifted gluing data} to be a $\check{C}$ech $1$-cocycle $(s_e)$ for the open cover $W := \{ W_\tau ~ | ~ \tau \in \P \}$ of $B$ and the sheaf $\iota_* \Lambda \otimes \mathbb{C}^*$.
\end{definition}
Not all open gluing data arises in this way from lifted gluing data -- see \cite[Prop 4.25]{Gross-Siebert3}. However, if it does, it
arises from unique lifted gluing data \cite[Thm 5.2]{Gross-Siebert3}. Thus we view the set of lifted gluing data as a subset of the set of open gluing data. The following result is \cite[Thm 5.4]{Gross-Siebert3}. %\footnote{The
%theorem makes also the converse statement using the dual
%intersection complexes; working polarized as we do here, imposes a
%codimension one constraint, see the definition of $A_\PP$ in
%\cite{GHS}, \S{A.2}.}
\begin{theorem}
\label{thm: lifted gluing}
In the case of positive and simple singularities, a polarized toric
log Calabi-Yau space with given dual intersection complex
$(B,\P)$ along with a specified MPL function $\varphi$, is defined uniquely up to isomorphism by
lifted gluing data $s\in H^1 (B,\iota_*\Lambda
\otimes\CC^\times)$. 
\end{theorem}
Theorem \ref{thm: lifted gluing}, can also be formulated on the cone side, following \cite{Gross-Siebert}. This would amount to replacing the dual intersection complex with the intersection complex, and choosing gluing data in
\[  s\in H^1 (B,\iota_*\breve{\Lambda}
\otimes\CC^\times) \]
For a comprehensive discussion on how to define the log structure on a toric log Calabi--Yau from lifted open gluing data, see \cite[\S4]{Gross-Siebert3}.

\section{The Kato--Nakayama Space of a Log Space}
\label{par: The Kato Nakayama Space}
 To any log analytic space $(X,\M_X)$ we can functorially assign a topological space, defined by Kato and Nakayama, called the Kato-Nakayama space, denoted by $X^{KN}$, or by $X_{log}$ as in \cite{KN}. In this section we review the definition of Kato--Nakayama spaces, along with some examples. 
\begin{definition}
\label{def:KN}
Let $(X,\M_X)$ be a log analytic space. Then, the \emph{Kato-Nakayama space of $X$}, denoted by $X^{KN}$ is the topological space defined by
$$X^{KN}:=  \Hom(\Pi^{\dagger} \to X^{\dagger}) $$ where $\Pi^{\dagger}$ denotes the polar log point $\Spec\CC$ with the associated log structure $\shM_{\Pi}:=\RR_{\geq0} \times S^1$ and the map $\alpha_{\Pi}:S^1\times \mathbb{R}_{\geq0}\rightarrow \mathbb{C}$ given by $(e^{i\phi},r)\rightarrow re^{i\phi}$. 
\end{definition}
Expanding Definition \ref{def:KN}, as explained in \cite{AS} we obtain
$$X^{KN}=\{(x,\phi)|~~x\in X,~~ \phi \in \Hom(\M^{gp}_{X,x},S^1) , \phi (h)=\frac{h(x)}{|h(x)|} ~~ \mathrm{for~~ any~~} h \in \mathcal{O}_{X,x}^{\times}\}$$.

\begin{proposition}
\label{topology ok KN}
Let $(X,\shM_X)$ be a log analytic space. Assume $\shM_X$ is fine, and let $\beta: \shP \to \shM_X$ be a chart. Then the Kato-Nakayama space $X^{KN}$ is embedded into $X \times \Hom(\shP^{gp},S^1)$ as a closed subset.
\end{proposition}

\begin{proof}
%The Grothendieck group $\shP^{gp}$ is obtained by adding to $\shP $ the inverse of each element inside. 
Since $\shP$ is integral, the map $\beta: \shP \to \shM_X$ extends to a map $\beta: \shP^{gp} \to \shM_X^{gp}$ which we will again denote by $\beta$ by abuse of notation. Let $(x,\phi)\in X^{KN}$ so that $x \in X$ and $\phi \in \Hom(M^{gp}_{X,x},S^1)$ satisfying $\phi (h)=\frac{h(x)}{|h(x)|} ~~ \mathrm{for~~ any~~} h \in \mathcal{O}_{X,x}^{\times}$. Define \[\phi_{\shP}:= \phi \circ \beta \in \Hom(\shP^{gp},S^1),\] and let
\begin{eqnarray}
\nonumber
\Psi:X^{KN} & \lra & X \times \Hom(\shP^{gp},S^1) \\
\nonumber
(x,\phi) & \longmapsto & (x,\phi_{\shP})
\end{eqnarray}
We will show that $\Psi$ is an embedding onto
\[ S:= \{ (x,\phi) \in X \times \Hom(\shP^{gp},S^1) ~~|~~(\alpha \circ \beta(p))(x) =\phi(p) \cdot  |(\alpha \circ \beta(p))(x)| \}.\]
Clearly, $\operatorname{Im} \Psi \subseteq S$. We will show that for any $(x,\phi_{\shP}) \in S$, there there exists a unique $(x,\phi) \in X^{KN}$, such that $\Psi(x,\phi)=(x,\phi_{\shP})$.
For $\phi_{\shP} \in \Hom(\shP^{gp},S^1) \cong \Hom(\shP,S^1)$ and $\shM_X$ the log structure associated to $\shP$ we have the cartesian diagram
\[
\xymatrix{
        \beta^{-1}(\mathcal{O}_{X,x}^{\times}) \ar@{^{(}->}[r] \ar[d]^{\alpha \circ \beta} & \shP \ar[d] \\
       \mathcal{O}_{X,x}^{\times} \ar[r]       & \shM_{X,x}}
\] 
Moreover this diagram is co-Cartesian, and by the universal property of the fibered coproduct there exists a unique map $\phi \in \Hom(\shM_{X,x},S^1)$ making the following diagram commute
\[
\xymatrix@C=30pt
{\beta^{-1}(\mathcal{O}_{X,x}^{\times}) \ar@{^{(}->}[r] \ar[d]^{\alpha \circ \beta}&\shP\ar[d]\ar@/^/[ddr]^{\phi_{\shP}}&\\
\mathcal{O}_{X,x}^{\times} \ar[r]\ar@/_/[drr]^{\arg \circ \operatorname{ev}_x}&\shM_{X,x}
\ar@{-->}[rd]^{\phi}&\\
&&S^1
}
\]
The map $\phi$ defines a unique element in $\Hom(\shM_{X,x}^{gp},S^1)$ with $\phi(h)=\frac{h(x)}{|h(x)|}$ for $h \in \mathcal{O}_{X,x}^{\times}$. Hence, the result follows.
%Note that to stay consistent with most references by abuse of notation we denoted $(\alpha_X\circ \beta(p))(x)$ as $(\beta(p))(x)$. However, we will sometimes use the convention $(\alpha_X\circ \beta(p))(x)$, in the remaining part of this article.
\end{proof}

The following is an immediate corollary of Proposition \ref{topology ok KN}.
\begin{corollary}
Let $(X,\shM_X)$ be log space, and assume $\M_X$ is fine. Then, the space $X^{KN}$ consists of points $(x,\Phi) \in X \times \Hom(\shM, \RR_{\geq 0} \times S^1)$ where $\Phi$ fits into the following commutative diagram

\begin{displaymath}
    \xymatrix{
        \shM_X \ar[r]^{\Phi} \ar[d]^{\alpha_X} & \RR_{\geq0} \times S^1 \ar[d]^{\alpha_{\Pi}} \\
        \mathcal{O}_{X,x} \ar[r]^{ev_x}       & \CC }
\end{displaymath}
where $\alpha_{\Pi}$ denotes the structure homomorphism on the polar log point.
\end{corollary}

%The following is an immediate corollary of Proposition %\ref\ref{topology ok KN}.
%\begin{corollary}
%\label{Xkn homeomorphism}
%Let $X= \Spec \CC[\shM]$ be a fine log scheme. Then $X^{KN}$ is homeomorphic to $\Hom(\shM, \RR_{\geq 0} \times S^1)$.
%\end{corollary}

\begin{lemma}
\label{X toric}
Let $X$ be an affine toric variety, with associated fan $\sigma \subset M_{\RR}$, where $M\cong \ZZ^n$. Let $\shM_{(X,D)}$ denote the divisorial log structure on $X$ where $D$ is toric boundary divisor. Then, there is a homeomorphism 
\[ X^{KN} \cong \sigma \times T^n \]
where $T^n$ denotes the compact $n$-torus.
\end{lemma}

\begin{proof}
By \ref{toric fine}, the divisorial log structure is fine and there is a canonical chart 
\begin{eqnarray}
\nonumber
\alpha: \check{\sigma} \cap N & \lra & \CC [\check{\sigma} \cap N] \\
\nonumber
(n_1,\cdots,n_d) & \longmapsto & z_1^{n_1}\cdots z_d^{n_d}
\end{eqnarray}
for $(n_1,\cdots,n_d) \in \check{\sigma} \cap N$, so $d= \dim X$, where $\check{\sigma}$ is the dual cone to $\sigma$. Moreover, by \cite[Example 2.3.2]{FK}, there is a splitting 
\begin{equation}
\label{Eq: XKN}
X^{KN}\cong \Hom((\check{\sigma} \cap N),\RR_{\geq0} \times S^1)\cong \Hom((\breve{\sigma} \cap N),\RR_{\geq0} )\times \Hom((\breve{\sigma} \cap N), S^1).
\end{equation}
The second factor in the right hand side of \eqref{Eq: XKN} is isomorphic to the compact torus $\Hom(\ZZ^n, S^1) \cong T^n$. The first factor, $\Hom((\breve{\sigma} \cap N),\RR_{\geq0} )$ is the positive real locus $X_{\geq 0 }$ defined in \eqref{eq: positive real locus}, which is a real manifold with corners of dimension $n$ \cite[\S10]{Oda}. Moreover, from the discussion in \S\ref{Sect: toric varieties} it follows that there is a homeomorphism
$X_{\geq 0 } \cong \breve{\sigma}$ (see also \cite[\S4.2]{Fulton}). Hence, the result follows.
\end{proof}
Since a projective toric variety admits an affine chart build by affine toric varieties, we can generalise Lemma \ref{X toric} to projective toric varieties.
\begin{example}
\label{P2 divisorial}
Let $X= \mathbb{P}^1$ be endowed with the log structure defined by the toric boundary divisor $D=\{0,\infty\}$. Recall that the moment map image of $\mathbb{P}^1$ is $[-1,1]$.  Hence, $X^{KN}\cong [0,1] \times S^1$.
\end{example}

%\begin{figure}
 %   \centering
  %  \includegraphics{P1_pspdftex.eps}
   % \caption{The Kato--Nakayama space over $\PP^1$ with maps to $\PP^1$ and to $[-1,1]\subset \RR$}
    %\label{fig:my_label}
%\end{figure}

We next discuss the topology on the Kato--Nakayama space. Proposition \ref{topology ok KN} defines a natural topology on the Kato--Nakayama space of log spaces that admit a fine log structure. In general, when the log structure is not fine, we describe the topology on the Kato--Nakayama space as follows \cite{NO}. Let $(X,\M_X)$ be a log analytic space. Let
\begin{eqnarray}
\label{Eq:tau}    
    \tau_X: X^{KN} & \lra & X \\
    \nonumber
    (x,\phi) & \longmapsto & x
\end{eqnarray}
be the projection map, and corresponding
to each section $m$ of $\M_X$ on an open subset $U$ define the function
\begin{eqnarray}  
    \mathrm{arg}(m): \tau_X^{-1}(U) & \lra & S^1 \\
    \nonumber
    (x,\phi) & \longmapsto & \phi(m_x)
\end{eqnarray}
 We endow $X^{KN}$ with the weakest topology such that the functions $\tau_X$ and $\mathrm{arg}(m)$ as $m$ ranges over the local sections of $\M_X$ are continuous.

\begin{remark}
Let $(X,\M_X)$ be a log analytic space. Assume $\M_X$ satisfies the following property: if $m \in \M_X$, $u' \in \M_X^{\star}$ and
$m + u = m$, then $u'= 0$. Then, by \cite[Chapter V Prop 1.2.5]{Ogus}, the map $\tau_X: X^{KN} \to X$ defined in \eqref{Eq:tau} is surjective. In this case, for a point $x \in X$, 
\[\tau^{-1}_X(x) \cong T^r = (S^1)^r \] if ${\overline{\shM}}^{gp}_{X,x} \cong \ZZ^r$.
\end{remark}
\begin{example}
Let $X$ be a log space with trivial log structure. Then $X^{KN}$ is homeomorphic to $X$.
Note that the topology on $X$ and the topology on $X^{KN}$ are the same in this case. 
\end{example}

\begin{example}
Let $X= \mathbb{A}^1$, endowed with the divisorial log structure given by the closed point $\{0\}$. Then, ${\overline{\shM}}_{X,0}\cong \NN$, and for any $x\neq 0$ the stalk ${\overline{\shM}}_{X,x}$ trivial. By Lemma \ref{X toric} it follows that $X^{KN}\cong \RR_{\geq 0} \times S^1$ and the natural projection map $\tau_X: X^{KN}\to X$, defined as in \eqref{Eq:tau} satisfies
$$ {\tau}_{X}^{-1}(t)= \left\{
	\begin{array}{ll}
		t  & \mbox{if } t \neq 0 \\
		S^1 & \mbox{if } t = 0
	\end{array}
\right.              $$
So, $X^{KN}$ is homeomorphic to the oriented real blow up of $\mathbb{A}^1$ at the origin, as illustrated in Figure \ref{fig:oriented real blow up}. For details see \cite[Chapter V,\S1.2]{Ogus}. 

\begin{figure}
    \centering
    \includegraphics{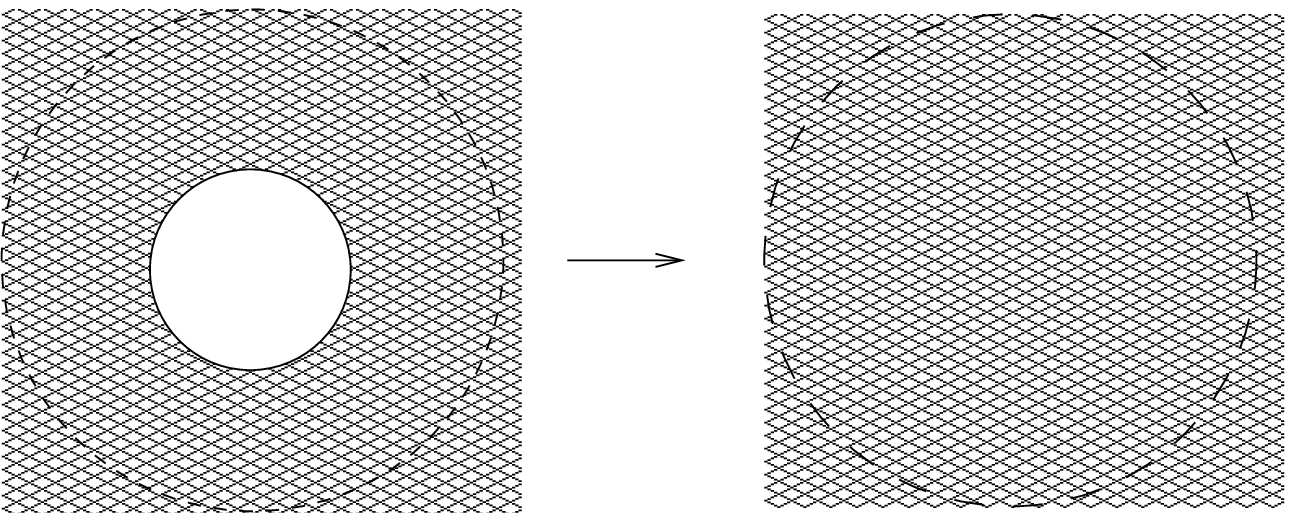}
    \caption{The oriented real blow-up of $\AA^1$ at the origin}
    \label{fig:oriented real blow up}
\end{figure}

\end{example}

\section{Torus Fibrations on $\mathcal{X}_t \cong \mathcal{X}_0^{KN}(\xi)$: the case of K3-surfaces}
\label{sec: example}
In this section we investigate the Kato--Nakayama space of a toric log Calabi--Yau, $\mathcal{X}_0^{KN}$ and describe a torus fibration on it. To do this, we will first consider the Kato--Nakayama spaces over each toric irreducible components over $\mathcal{X}_0$, and analyse how they glue together. For simplicity we will assume that the log structure on the central fiber is determined by trivial gluing data \cite[Remark 4.16]{AS}. The following proposition, which appears as Lemma $5.13$ in \cite{Gross-Siebert3}, is a crucial result, which explains how the log structure on $\mathcal{X}_0$ interacts with log structures of its irreducible components.  
\begin{proposition}
\label{splitting}
Let $\pi: \shX \to \Spec R$ be a toric degeneration and let $(B,\P)$ denote the dual intersection complex $B$ endowed with a polyhedral decomposition $\P$. For $\tau\in\P$, let $q_{\tau}:X_{\tau}\rightarrow \mathcal{X}_0$ be the
natural inclusion map, and we will also write here $q_{\tau}$ for the
restriction $q_{\tau}: X_{\tau}\setminus q_{\tau}^{-1}(Z)\rightarrow
\mathcal{X}_0\setminus Z$ where $Z$ denotes the log singular locus.  We write $\M_{\tau}:=q_{\tau}^*\M_{\mathcal{X}_0\setminus Z}$,
the pull-back of the log structure. For any $\tau\in\P$, denote by $D_{\tau}\subseteq X_{\tau}$
the toric boundary of $X_{\tau}$. Then, there is a natural exact sequence
\[
0\lra \M_{(X_{\tau},D_{\tau})}^{\gp}
\lra \M_{\tau}^{\gp}\lra \breve{\Lambda}_{\tau}\oplus\ZZ
\lra 0
\]
on $X_{\tau}\setminus q_{\tau}^{-1}(Z)$, where $\M_{(X_{\tau},D_{\tau})}$ is the divisorial log structure on $X_{\tau}$, and $\breve{\Lambda}_{\tau}$ denotes the integral cotangent lattice on $\tau$. This exact sequence splits,
and the splitting is canonical if  $\dim\tau=0$, that is the dimension of the orbit $X_{\tau}$ is maximal.
\end{proposition} 

Now, let $(\mathcal{X}_0,\M_{\mathcal{X}_0})$ be a toric log Calabi--Yau space, formed by unions of toric varieties glued along toric strata. Consider the intersection complex $B$, which is obtained from the dual intersection complex by a discrete Legendre transform \cite[\S4]{Gross-Siebert4}. Note that $B$ is topologically the same manifold as the dual intersection complex. However, the affine monodromy differs \cite[\S3]{AP}. Let $\Delta \subset B$ be the discriminant locus. Then, on the log singular locus $Z:=\mu^{-1}(\shA)$, where $\shA$ denotes the amoeba image of the singular locus $\Delta$, the log structure $\shM_{\mathcal{X}_0}$ is not fine, and there does not exist a chart. Indeed analysing the Kato--Nakayama space of $\mathcal{X}_0$ in this case is rather challenging. However, we can easily describe the Kato--Nakayama space away from the log singular locus as follows: Let $Y_i$ be a toric irreducible component of $\mathcal{X}_0$, and let $\mathring{Y_i}:=Y_i\setminus Z$ with the induced log structure
\[\restr{\shM}{\mathring{Y_i}}:=\restr{\shM_{\shX_0}}{\mathring{Y_i}} \]
Note that the induced log structure $(Y_i,\restr{\shM}{\mathring{Y_i}})$ is fine on $\mathring{Y_i}$. So, we can define a chart for $\restr{\shM}{\mathring{Y_i}}$. Assume $\mathcal{P}\to \shM_{(Y_i,D_i)} $ is a chart for $(Y_i,\shM_{(Y_i,D_i)}$ with the divisorial log structure given by the toric boundary divisor $D_i \subset Y_i$. Then, it follows from Proposition \ref{splitting} that a chart for the log structure $\restr{\shM}{\mathring{Y_i}}$ is given by $\shP \oplus \NN$. Let $\sigma_i$ be the moment map image of $Y_i$.  From Lemma \ref{X toric}, it follows that there is a canonical homeomorphism $\mathring{Y_i}^{KN} \cong (\sigma_i \setminus \Delta) \times \Hom({\Lambda}_{\sigma_i},S^1) \times S^1$,
where the latter $S^1$-factors comes due to the fact that a chart for $\restr{\shM}{\mathring{Y_i}}$ is given by $\shM \oplus \NN$. Hence, each copy $Y_i^{KN}$ is a trivial $T^{n+1}$ fibration over $\sigma_i \setminus \Delta$. This induces an $T^{n+1}$ fibration $\tilde{\mu}: \shX_0^{KN}\to B$ given by the composition $\mu \circ \tau_{\shX_0 \setminus Z}$, 
\[
\xymatrix@C=30pt
{{(\shX_0 \setminus Z)}^{KN} \ar@{-->}[rrd]^{\tilde{\mu}} \ar@{->}[rr]^{\tau_{\shX_0 \setminus Z}}  & & (\shX_0 \setminus Z) \ar[d]^{\mu}&\\
& & B \setminus \Delta
}
\]
where $\tau_{\shX_0 \setminus Z}$ is the projection map defined as in \eqref{Eq:tau}, and by abuse of notation we denote the restriction of the generalized moment map $\mu: \shX_0 \to B$ to $\shX_0 \setminus Z$ again by $\mu$.
\begin{remark}
\label{restricted KN}
Let $O^{\dagger}= (\Spec \CC, \NN \oplus \CC^{\star})$ be the standard log point. Then, $(O^{\dagger})^{KN} \cong S^1$. For for any point $\xi \in S^1$, we have the following commutative diagram
$$
    \xymatrix{
        \shX_0^{KN}  \ar[d] & \shX_0^{KN}(\xi)  \ar[l] \ar[d] \\
        S^1=(O^{\dagger})^{KN}     & \xi \ar[l] }
$$
where the horizontal arrows are inclusions, and $\shX_0^{KN}(\xi)$ denotes the fiber of the  Kato-Nakayama space $\shX_0^{KN}$ over $\xi$. We will use the analogous notational convention for the Kato-Nakayama space over the irreducible components $Y_i^{KN}$. So we obtain $\shX_0^{KN}(\xi)$ as a $T^n$ bundle over $B$, away from the discriminant locus $\Delta \subseteq B$. 
\end{remark}

Recall that the fiber of the Kato--Nakayama space $\shX_0^{KN}(\xi)$ is homeomorphic to the general fiber of the smoothing of the toric log Calabi--Yau space $\shX_0$ by Theorem \ref{Thm: AS}. To illustrate this on a concrete example, we will investigate a degeneration of $K3$-surfaces.
\begin{example}
Let $\shX=\{ X_0X_1X_2X_3+t\cdot f_4=0 \}\subset \mathbb{P}^3\times\mathbb{A}^1_t$ for a generic degree $4$ homogenous polynomial $f_4$. This defines a toric degeneration, where the general fiber, $\shX_t$, over $t \neq 0$ is a K$3$ surface and the central fibre $\shX_0=(X_0X_1X_2X_3=0) \subseteq \mathbb{P}^3$ consists of $4$ copies of the projective space $\mathbb{P}^2$, meeting pairwise along $\mathbb{P}^1$. Then the intersection complex $B$ associated to $\shX_0$, as a topological manifold, is the boundary of a tetrahedron. We will in a moment describe the discriminant locus of the integral affine structure with singularities on $B$, as the image of the log singular locus under the generalized (abstract) moment map $\mu: \shX_0 \to B$, defined as in \cite[Definition 2.3]{AS}. To carry out this computation concretely, assume 
\begin{equation}
\label{Eq f4}
f_4 = X_0^4+X_1^4+X_2^4+X_3^4 - 7(X_0^2X_1^2+X_0^2X_2^2+X_0^2X_3^2+X_1^2X_2^2+X_1^2X_3^2+
X_2^2X_3^2)
\end{equation}
Then, over each edge of $B$, the inverse image of $\mu$ is homeomorphic to a copy of the projective line $\mathbb{P}^1$. Without loss of generality, consider \[\mathbb{P}^1[X_0:X_1]:=\{ (X_0:X_1:X_2:X_3) \in \mathbb{P}^2~~|~~X_2=X_3=0    \}.\] Inserting $X_2=X_3=0$ in Equation \eqref{Eq f4}, we obtain
$$X_0^4+X_1^4 - 7(X_0^2X_1^2)=0$$
Dehomogenising with $X_1$ and inserting $x=X_0 / X_1$ we get 
$$x^4-7x^2+1=(x^2-3x+1)(x^2+3x+1)=0$$
Thus, over each edge we get four real roots 
\[ \left[ \frac{3+\sqrt{5}}{2}:1:0:0 \right],\left[\frac{3-\sqrt{5}}{2}:1:0:0\right],\left[\frac{-3+\sqrt{5}}{2}:1:0:0\right],\left[\frac{-3-\sqrt{5}}{2}:1:0:0\right].\] Doing the analogous computation for the other edges, we obtain in total $24$ points. On an affine chart, around each such point the local equation is given analogously as in \cite[Example 3.20]{Gross}, and therefore the log structure is not fine around these points. Hence the discriminant locus $\Delta \subset B$, consists of $24$ points as illustrated in Figure \ref{Fig: tetrahedron}.

\begin{figure}
\center{\input{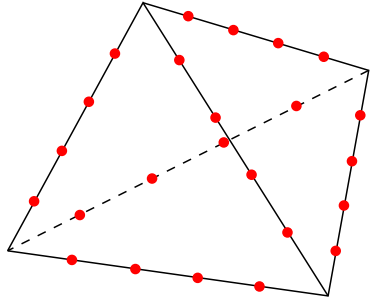}}
\caption{$(B,\P)$ for a quartic K3 surface}
\label{Fig: tetrahedron}
\end{figure}
By Remark \ref{restricted KN}, away from $\Delta$, 
\[ (\mathcal{X}_0 \setminus Z)^{KN}(\xi) \to B \setminus \Delta\]
is an $S^1 \times S^1$-fibration over $B \setminus \Delta$. Over each point $p \in \Delta$, the  fiber of the Kato-Nakayama space $(\shX_0)^{KN}$ is a nodal elliptic curve, by the analysis done in Examples $2.9$ and $2.10$ in \cite{AS}. This defines a singular torus fibration 
\[ (\mathcal{X}_0)^{KN}(\xi) \to B \]
on a $K3$-surface $\shX_t \cong (\mathcal{X}_0)^{KN}(\xi)$.

\end{example}

\end{document}